\RequirePackage{fix-cm}
\documentclass[smallcondensed]{svjour3}     
\smartqed  
\usepackage{latexsym}
\usepackage{ragged2e}
\usepackage{amssymb,amsfonts,fancyhdr}
\usepackage{picinpar,subfigure}
\usepackage{float}
\usepackage{pstricks}
\usepackage{fancyvrb}
\usepackage{cite}
\usepackage{mathrsfs}
\usepackage{multirow}
\usepackage{rotating,makecell}
\usepackage{caption}

\usepackage{amsmath}
  \usepackage{paralist}
  \usepackage{graphics} 
  \usepackage{epsfig} 
\usepackage{graphicx}  \usepackage{epstopdf}
 \usepackage[colorlinks=true]{hyperref}

\newcommand{\dd }{\mathrm{d}}

 \journalname{Zeitschrift f\"{u}r angewandte Mathematik und Physik}
\begin{document}
\allowbreak
\title{Delta Shock as Free Piston in Pressureless Euler Flows\thanks{This work is supported by the National Natural Science Foundation of China under Grants No. 11871218, No. 12071298, and by the Science and Technology Commission of Shanghai Municipality under Grant No. 18dz2271000.}
}

\titlerunning{Delta Shock as Free Piston }

\author{Le Gao        \and Aifang Qu \and Hairong Yuan
}

\authorrunning{Gao, Qu and Yuan} 

\institute{Le Gao \at
              Center for Partial Differential Equations, School of Mathematical Sciences, East China Normal University, Shanghai 200241, China   \\
              \email{51195500061@stu.ecnu.edu.cn}
           \and
          Aifang Qu \at
              Department of Mathematics, Shanghai Normal University, Shanghai  200234,  China\\
    \email{afqu@shnu.edu.cn}
    \and
 Hairong Yuan \at
              School of Mathematical Sciences and Shanghai Key Laboratory of Pure Mathematics and Mathematical  Practice, East China Normal University, Shanghai 200241, China  \\
              \email{hryuan@math.ecnu.edu.cn}
}

\date{Received: date / Accepted: date}

\maketitle

\begin{abstract}
We establish the equivalence of free piston  and delta shock, for the one-space-dimensional pressureless compressible Euler equations. The delta shock appearing in the singular Riemann problem is exactly the piston that may move freely forward or backward in a straight tube, driven by the pressureless Euler flows on two sides of it in the tube.    This result not only helps to understand the physics of the somewhat mysterious delta shocks, but also provides a way to reduce the fluid-solid interaction problem, which consists of several initial-boundary value problems coupled with moving boundaries, to a simpler Cauchy problem.  We show the equivalence from three different perspectives. The first one is from the sticky particles, and derives the ordinary differential equation (ODE) of the trajectory of the piston by a straightforward application of conservation law of momentum, which is physically simple and clear. The second one is to study a coupled initial-boundary value problem of pressureless Euler equations, with the piston as a moving boundary following the Newton's second law. It depends on a concept of Radon measure solutions of initial-boundary value problems of the compressible Euler equations which enables us to calculate the force on the piston given by the flow. The last one is to solve directly the singular Riemann problem and obtain the ODE of delta shock by the generalized Rankine-Hugoniot conditions. All the three methods lead to the same ODE. 

\end{abstract}

\section{Introduction}\label{sec1}

The piston problem, serving as an important typical physical model in mathematical fluid dynamics, has been studied extensively in the past decades; See, for example,  \cite{courant1976supersonicflowandshockwaves, chen2005amultidimensionalpiston, chen2008globalexistenceofshock,Chen2019Piston,Liu1978freepiston} and { references} therein. There are many different situations. The classical one {is that}
the piston moves with given speed and drives the gas in front of it in a tube. In this paper, we consider the free piston problem, for which the piston is a finite body, with force given by gas which is filled in the tube, both in front of the piston and behind of it. Then the piston's motion is determined by the gas flow, which in turn affects the flow field.
Liu \cite{Liu1978freepiston} had studied {the weak} entropy solutions of free piston problem for the $p$-system, i.e., the barotropic compressible  Euler equations formulated  in the Lagrangian coordinates. We will study instead a free piston problem for the following pressureless Euler equations:
\begin{equation}\label{1.1}
\left \{
\begin{split}
&\partial_{t}\rho+\partial_{x}(\rho u)=0,\\
&\partial_{t}(\rho u)+\partial_{x}(\rho u^{2})=0,\\
\end{split} \right.
\end{equation}
where $t\geq0$ is a time variable, and $x\in \Bbb R$ is a space variable; $\rho$,  $u$ represent respectively the density of mass, and the velocity of the gas. {Hereafter}, for convenience, we also write $U=(\rho, u)^\top$, which is a vector-valued function of $(x,t)$. These equations are derived from the integral equations describing conservation laws of mass and momentum, under the assumption that the flow field is quite smooth. In \cite{Qu2020Hypersonic}, it has been shown that under suitable scaling, the pressureless flow is actually a hypersonic limiting flow.  If the speed of the piston is very large comparing to the flow, in the high Mach number limit, polytropic gas may also be considered as pressureless flow; See \cite{Qu2020High}. We also recommend \cite{lizhangyang} for a very nice introduction to the background and theory of the system \eqref{1.1}.

One motivation of studying free piston problem of pressureless Euler flow is to present a new and clear physical interpretation of delta shocks, which are measures concentrated on space-time curves modeling concentration of mass and momentum. It has been a focused subject of research in the past thirty years from the theoretical perspective of systems of hyperbolic conservation laws, since the work \cite{Tan1994Two} of Tan and Zhang.
The solvability and wave structures involving delta shocks of Riemann problems for pressureless gas and other hyperbolic conservation laws in one or two space dimensions have been thoroughly analyzed. See, for example \cite{Li1998The,Sheng1999The,Yang2012New,lizhangyang}. There are also many works devoted to showing the reasonability of delta shocks, such as  \cite{Chen2003Formation,Li2001Note} for the formation of delta shocks when pressure or temperature of polytropic gases decreases to zero.
Up to now, several mathematical theories on Cauchy problems of pressureless Euler equations have  been established by different methods for different initial data, see, for instance,   \cite{Bouchut1994On,  E1996Generalized, Huang2001Well,liw2003,chenglizhang1997}.

Unlike classical elementary waves of {the Euler} equations such as shock waves or rarefaction waves, delta shocks are introduced mathematically to solve certain Riemann problems and its physical relevance is always worth of investigation. In a private communication {with Professor} Jiequan Li, he suggested {that a} natural way  to understand delta shock is to regard it as a free piston which sticks all the gas particles impinging on it. Obviously, this requires to take into account of the interactions of the piston and the flow, thus a rigorous definition of Radon measure solutions of initial-boundary value problems of compressible Euler equations is indispensable for the mathematical analysis. However, there is no any  work on Radon measure solutions of boundary value problems of Euler equations priori to Qu, Yuan and Zhao's paper \cite{Qu2020Hypersonic}. In subsequent papers \cite{Qu2020High,Qu2020Measure,Jin2019On,jinquyuancpaa,Quwangyuancms,quyuanjde,jinquyuan2021riemann}, the authors developed {the ideas} in \cite{Qu2020Hypersonic} and solved explicitly some typical problems, including hypersonic flow passing wedges/cones, {piston problems, and Riemann problems, for polytropic gases, Chaplygin gas and pressureless gas. Particularly, the fundamental Newton's sine-squared pressure law and Newton-Busemann formula for hypersonic aerodynamics are rigorously proved. These works have fully demonstrated the great role that measure solutions play in the study of hyperbolic conservation laws, and a proper definition of measure solutions is of vital importance as lots of physical problems can only be solved with its help. The success enables us to treat the free piston problem here.

We demonstrate that delta shock could be regarded as a free piston in pressureless flow by treating the problem via three different approaches. (To this purpose, it is enough to assume that the flow ahead of and behind  the piston are both constant initially.) The first one is regarding the pressureless flow as sticky particles and deriving the ODE of the position of the piston directly from conservation of momentum. It has the merit to be physically straightforward and short. The second one is to study the Radon measure solutions of two initial-boundary value problems that are coupled with a common moving boundary. The two initial-boundary value problems model the pressureless flows in the tube,  and the piston, i.e., the moving boundary, is solved by Newton's second law, with forces calculated from Radon measure solutions describing the flow fields. The third one is to solve a singular Riemann problem of pressureless Euler equations, whose delta shock solution is given by  the generalized Rankine-Hugoniot conditions, which in turn are derived from the definition of Radon measure solutions for a Cauchy problem. All these approaches lead to the same ODE, which manifests the equivalence of free piston and delta shock. For completeness, in the Appendix we list the results for the six cases of the free piston problem without detailed computations, which are the same as those for delta shocks presented in \cite[Section 3.4]{lizhangyang}.

We remark that except the solution of the free piston problem and the new  interpretation of delta shocks as free pistons, a notable novelty of this paper is the fact that a solid-fluid interaction problem can be dealt in a unified way by solving an initial value problem of the compressible Euler system in the larger class of admissible solutions such as Radon measures. It is worth pointing out that although the latter two methods look somewhat complicated for the simple model problem studied in the paper, they have the merit that could be extended to deal with similar problems with general initial data and structures.

The rest of the paper is organized as follows. In section \ref{s2}, we formulate the free piston problem, the singular Riemann problem, and state main theorems for two  fundamental  cases. Section \ref{sec3} is devoted to deriving the differential equations of the piston's trajectory and delta shocks, and thus verify the equivalence of a free piston and a delta shock. 
Finally, the Appendix  presents the explicit expressions of the piston's motion/delta shock.

\section{The free piston problem, Riemann problem and main results}\label{s2}

In this section, after formulation of the free piston problem and singular Riemann problem for the pressureless Euler equations, we present the main results of the paper, which will be proved in the next section.

\subsection{The free piston problem}\label{s2.1}
System \eqref{1.1} models pressureless gas in a thin and straight tube with constant cross-sections, identified with the horizontal $x$-axis. There is a solid body (piston) with thickness $l>0$ and mass $m_{0}>0$ in the tube, thus divides the tube into two parts, each of which is filled with uniform pressureless gas, with density $\rho_1$ and velocity $u_1$ in the left part, and density $\rho_2$, velocity $u_2$ in the right part.  
Without loss of generality, we assume that the cross-section of the piston as well as the tube is of unit area.  The  piston  is driving by the force exerting on it by the pressureless flow, whose particles stick on the piston once they meet it, via perfectly inelastic collisions.  By the Galilean principle of relativity, we may assume that initially the  piston's left side locates at the origin, with speed $u_0$.
Thus the initial condition of the flow is
\begin{equation}\label{1.2}
U(x,0)=
\begin{cases}
 (\rho_{1},u_{1})^{\top},&\mbox{if~}x<0, \\
 (\rho_{2},u_{2})^{\top},&\mbox{if~}x>l,
\end{cases}
\end{equation}
where $u_{1},~\rho_{1},~u_{2},~\rho_{2}$ are all given constants, with $\rho_1, \rho_2$ being nonnegative.

The location of the piston is an unknown function of time $t$, given by
$x=x_{1}(t)$  (resp. $x=x_{2}(t)$),  which represents the position of the piston's left  (resp. right) side. We assume the piston is solid, hence there is no any change of its length for all the time. Therefore there holds
\begin{equation}\label{1.3}
\begin{split}
x_{2}(t)=x_{1}(t)+l,\ \ \ \forall t\ge0; \qquad x_1(0)=0, \qquad x'_1(0)=u_0.
\end{split}
\end{equation}
On the surface of the piston, we propose the natural condition for perfectly inelastic collisions
 \begin{equation}\label{2.2}
u(x_1(t),t)=x_{1}'(t)=x_{2}'(t)=u(x_2(t),t), \quad\forall\, t\ge0,
\end{equation}
which means that particles on the piston's surface will move along with the piston.
The piston suffers inertial forces caused by the flows hitting the piston, and we set $w_{p}^{1}(t),~w_{p}^{2}(t)$ be the inertial forces (pressure) on the left and right side of the piston respectively.
By Newton's second law, there should hold
\begin{align}\label{eq5}
x''_1(t)=\frac{1}{m_0}\big({w_p^1(t)-w_p^2(t)}\big).
\end{align}
Denote the space-time region on the left and right side of the piston respectively by
\begin{align}\label{eq6add}
\Omega_{1}\doteq\{(x,t)\in \Bbb R^{2}: \ \ t>0, x<x_{1}(t)\},
\end{align}
and
\begin{align}\label{eq7add}
\Omega_{2}\doteq\{(x,t)\in \Bbb R^{2}: \ \ t>0, x>x_{2}(t)\}.
\end{align}
The free piston problem considered in this paper is to determine the motion of the piston $x=x_1(t)$ from \eqref{eq5}, with state $U_1$ in $\Omega_1$ and state $U_2$  in $\Omega_2$, satisfying the equations \eqref{1.1}, subjected to the boundary conditions \eqref{1.2}--\eqref{2.2}.

\subsection{The (singular) Riemann problem}\label{s2.2}
We also formulate the following measure $\bar{\varrho}$ and function $\bar{u}$ that is $\bar{\varrho}$-measurable, as initial data for \eqref{1.1}:
\begin{align}\label{eq6}
  \bar{\varrho}=\rho_1\mathcal{L}^1\lfloor\{x<0\}+m_0\delta_0+\rho_2\mathcal{L}^1\lfloor\{x>0\},  \bar{u}=u_1\mathsf{I}_{\{x<0\}}+u_0\mathsf{I}_{\{x=0\}}+u_2\mathsf{I}_{\{x>0\}}.
\end{align}
Here $\mathcal{L}^1$ is the Lebesgue measure on the real line $\mathbb{R}$, $\delta_0$ is the Dirac measure supported at the origin, and $\mathsf{I}_A$ is the indicator function of a set $A$, while $m\lfloor A$ is the measure obtained by restricting a measure $m$ on an $m$-measurable set $A$. We call the Cauchy problem \eqref{1.1}\eqref{eq6} as a singular Riemann problem, since if $m_0=0$ and $u_0=0$, it is reduced to the well-known Riemann problem. Formally, taking the piston as a point (that means, using the partition $\big\{\{x\}_{x\not\in[0,l]\}}, [0,l]\big\}$ of $\mathbb{R}$ to obtain a quotient topological space, which is homeomorphic to $\mathbb{R}$), the free piston problem looks the same as the singular Riemann problem. We will show they are actually the same. For this purpose, we {need} the following concept of Radon measure solutions for the singular Riemann problem \eqref{1.1}\eqref{eq6}.  Firstly, we recall that by Riesz representation theorem,  $\mathscr{M}(\mathbb{R})$,  the space of Radon measures on $\mathbb{R}$, is dual to $C_c(\mathbb{R})$, space of continuous functions with compact supports. Equipping $\mathscr{M}(\mathbb{R})$ with the weak* topology, $C([0,+\infty); \mathscr{M}(\mathbb{R}))$ is the space of continuous mappings from $[0,+\infty)$ to $\mathscr{M}(\mathbb{R})$. As usual, we use $\langle m,\phi\rangle$ to denote the action of a measure $m$ on a test function $\phi$.

\begin{definition}\label{def1}
Let $\varrho,~m,~n\in C([0,+\infty); \mathscr{M}(\mathbb{R}))$ with $\varrho(t)$
nonnegative for all $t\ge0$.
We call $(\varrho,u)$  a (Radon) measure solution to the singular Riemann problem {\eqref{1.1}\eqref{eq6}}, provided that the following are valid:
\par i) for any $\phi(x,t)\in C_{c}^{1}(\Bbb R^2)$ (continuously differentiable functions with compact supports),  there hold
\begin{equation}\label{2.43}
\begin{split}
 \int_{0}^{+\infty}\langle \varrho(t),\phi_t(\cdot,t)\rangle\, \dd t&+\int_{0}^{+\infty}\langle m(t),\phi_x(\cdot,t)\rangle\,\dd t+\langle \varrho(0),\phi(\cdot,0)\rangle =0,\\
\int_{0}^{+\infty}\langle m(t),\phi_t(\cdot,t)\rangle\,\dd t&+\int_{0}^{+\infty}\langle n(t),\phi_x(\cdot,t)\rangle\,\dd t
+\langle m(0),\phi(\cdot,0)\rangle=0,
\end{split}\end{equation}
where $\varrho(0)=\bar{\varrho},$ and $m(0)=\bar{u}\bar{\varrho}=\rho_1u_1\mathcal{L}^1\lfloor\{x<0\}+m_0u_0\delta_0+\rho_2u_2\mathcal{L}^1\lfloor\{x>0\}$.
\par ii) $\forall\, t\ge0$, $m(t)$ and $n(t)$ are aboslutly continuous with respect to $\varrho(t)$, with the  Radon-Nikodym derivatives satisfying
\begin{equation}\label{2.44}
\begin{split}
&u(x,t)\doteq\dfrac{\dd m(t)(x)}{\dd\varrho(t)(x)},\quad u(x,t)^2=\dfrac{\dd n(t)(x)}{\dd\varrho(t)(x)}.
\end{split}
\end{equation}
\end{definition}
\qed

This definition is consistent with the familiar integrable weak solutions when $\varrho(t)$ is
absolutely continuous with respect to the Lebesgue measure $\mathcal{L}^1$ for all $t\ge0$, with its Radon-Nikodym derivative $\rho(x,t)$. We note that $u(x,t)$ does not make sense on a $\varrho$-null open set (i.e. interior of vacuum).

Let $x=x(t)$ be a Lipschitz function for $t\ge0$. Radon measure solutions to the singular Riemann problem {\eqref{1.1}\eqref{eq6}} of the form
\begin{align}\label{eq9}
&\varrho(t)=\tilde{\rho}_1\mathcal{L}^{1}\lfloor\{x: x<x(t)\}+\tilde{\rho}_2\mathcal{L}^{1}\lfloor\{x: x>x(t)\} +\alpha(t)\delta_{x(t)},\\
&u(x,t)=\tilde{u}_1\mathsf{I}_{\{x: x<x(t)\}}(x) +\tilde{u}_2\mathsf{I}_{\{x: x>x(t)\}}(x)+x'(t)\mathsf{I}_{\{x(t)\}}(x)\label{eq10}
\end{align}
is called a delta shock solution, where $\tilde{\rho}_{1,2}\ge0$ and $\tilde{u}_{1,2}$ are constants; $\alpha(t)$ is a nontrivial nonnegative function for $t\ge0$, and $\delta_{x(t)}$ is the Dirac measure on $\mathbb{R}$, supported at the point $\{x(t)\}$. The curve $x=x(t)$ is also called {a delta} shock. It satisfies the over-compressing entropy condition \cite{Bouchut1994On}, provided that
\begin{align}\label{eq11}
  \tilde{u}_1\ge x'(t)\ge \tilde{u}_2, \quad\forall t\ge0,
\end{align}
which means no particle releasing from the delta shock. Such delta shock solutions are called entropy solutions.

\subsection{Equivalence of delta shocks and free pistons}\label{sec23}

It is obvious that there are six cases for the free piston problem or singular Riemann problem, depending on the relations of initial velocities:
\begin{equation*}
  \begin{split}
     & {\rm Case ~1}: ~u_2<u_0<u_1; \quad {\rm Case~}  2:  u_0<u_1\leq u_2 ; \quad {\rm Case~}  3:  u_0<u_2<u_1 ; \\
       & {\rm Case~}  4: ~ u_1<u_2< u_0;\quad {\rm Case~}  5: u_2\leq u_1<u_0;\quad {\rm Case~}  6:  u_1\leq u_0\leq u_2 .
  \end{split}
\end{equation*}
For Case 1, gas concentrates on both sides of the piston and no vacuum presents for any $t>0$. The velocity of the piston will be in the interval $(u_2,u_1)$ and thus satisfies the entropy condition \eqref{eq11}. For Case 2, there is gas concentrated on the left side of the piston and no particle sticks to the piston on the right side of it, where there is a vacuum  between the piston and the uniform gas. The velocity of the piston will be in the interval $(u_0,u_1)$ and thus satisfies \eqref{eq11} in a weak sense (with $\tilde{u}_2$ not defined). These two cases are fundamental as the other cases  could be reduced to either one or a combination of them by a Galilean transformation and a change of orientation of the coordinates. More precisely, for Case 3, at first, the concentration of gas only happens on the left side of the piston and there is a vacuum domain between the piston and the gas on the right. As time goes on, the piston speeds up and will catch up with the gas on the right at some time $t_1$. (Notice that delta shock solutions of Riemann problems might not be self-similar). After that time  the problem is reduced to that of Case 1. By changing $x$ to $-x$, thus $u$ to $-u$, Cases 4 and 5 can be reduced to Cases 2 and 3 respectively. Case 6 is trivial, in the sense that there is no gas concentrated on either side of the piston. There are vacuum near both sides of the piston and the velocity of the piston will be unchanged.

We have the following results on  the trajectory of piston for  Case 1: $ u_2<u_0<u_1$ and Case 2: $u_0<u_1\leq u_2$.
\begin{theorem}\label{t2.1}
For the free piston problem \eqref{1.1}-\eqref{eq5}, with the mass of the piston $m_0>0$ and length $l>0$, if $u_2< u_0 <u_1$, then the trajectory {of the} piston's left-side $x=x_{1}(t)$ satisfies the following equations
\begin{equation}\label{e2.1}
\left \{
\begin{split}
&x_{1}''(t)+\dfrac{( {\rho_1}- {\rho_2})x_{1}'(t)-2( {\rho_1} {u_1} - {\rho_2} {u_2})}{( {\rho_1}- {\rho_2})x_{1}(t)-( {\rho_1} {u_1} - {\rho_2} {u_2})t-m_{0}}x_{1}'(t)\\
&\qquad\qquad+\dfrac{ {\rho_1} {u_1}^{2}- {\rho_2} {u_2}^{2}}{( {\rho_1} - {\rho_2})x_{1}(t)-( {\rho_1} {u_1}- {\rho_2} {u_2})t-m_{0}}=0,\\	&x_{1}(0)=0,~x_{1}'(0)=u_0.
\end{split} \right.
\end{equation}
On the left (resp. right) side of the piston, the flow is of constant state $(\rho_1, u_1)^{\top}$ (resp. $(\rho_2, u_2)^{\top}$) given by the initial data.
\end{theorem}

\begin{theorem}\label{t2.2}
  Under the assumptions of Theorem \ref{t2.1} but $u_0<u_1\leq u_2$, the trajectory function
  $x_{1}(t)$ solves
\begin{equation}\label{e2.2}
\left \{
\begin{split}
&x_{1}''(t)+\dfrac{\rho_{1}x_{1}'(t)-2\rho_{1}u_{1}}{\rho_{1}x_{1}(t) -\rho_{1}u_{1}t-m_{0}}x_{1}'(t)+\dfrac{\rho_{1}u_{1}^{2}}{\rho_{1}x_{1}(t)-\rho_{1}u_{1}t-m_{0}}=0,\\	&x_{1}(0)=0,~x_{1}'(0)=u_0.
\end{split} \right.
\end{equation}	
On the left side of the piston, the flow is of constant state $(\rho_1, u_1)^{\top}$ given by the initial data. On the right side of the piston, there is a vacuum $\{x: x_1(t)+l<x<u_2t+l\}$, and for $x>u_2t+l$, the flow is of constant initial state $(\rho_2, u_2)^\top$.
\end{theorem}

\begin{remark}\label{rem1}
From these two theorems we see that the trajectory of the piston relies on its mass $m_0$ rather than its thickness $l$. Therefore mathematically regarding the piston as a moving mass-point  in the fluid is reasonable.
\end{remark}
\begin{remark}\label{rem2new}
\eqref{e2.2} could be considered as a special case of \eqref{e2.1} for $\rho_2=0$, since for Theorem \ref{t2.2}, the state near the right side of the piston is vacuum as indicated  above.
\end{remark}

\begin{theorem}\label{thm2.3}
For initial data satisfying Cases 1 and 2 respectively,  the singular Riemann problem  \eqref{1.1}\eqref{eq6} has delta shock solutions with delta shock $x=x_1(t)$ given respectively by  \eqref{e2.1} and \eqref{e2.2}, and for Case 1, $\alpha(t)=(\tilde{\rho}_2-\tilde{\rho}_1)x_{1}(t)-(\tilde{\rho}_2\tilde{u}_2-\tilde{\rho}_1\tilde{u}_1)t +m_{0},$ $\tilde{\rho}_{1,2}=\rho_{1,2}, \tilde{u}_{1,2}=u_{1,2}$; for Case 2, $\alpha(t)=-{\rho}_1x_{1}(t)+{\rho}_1{u_1}t +m_{0},$ $\tilde{\rho}_{1}=\rho_{1}, \tilde{u}_{1}=u_{1},$ and $\tilde{\rho}_2=\begin{cases} 0, &\mbox{if~}x_1(t)<x<u_2t,\\
\rho_2, & \mbox{if~}x\ge u_2t,
\end{cases}$ while $\tilde{u}_2$ is not defined for $x_1(t)<x<u_2t$, and equals to $u_2$ for $x\ge u_2t$.
\end{theorem}
\begin{remark}
The states on the left/right hand side of the piston (delta shock) can be easily verified due to the simplicity of pressureless Euler flows as there is no acoustic wave in the fluids. So combining the observations in Remark \ref{rem1}, we see Theorem \ref{thm2.3} means that the free piston and the delta shock are actually the same thing, after viewing the interval $[x_1(t),x_1(t)+l]$ as a single point in a quotient topology of $\mathbb{R}$ suggested before. Therefore the  solution of the free piston problem could be easily obtained from  that of the singular Riemann problem, which may be found in \cite[Section 3.4.2, pp. 67-72]{lizhangyang}. See also the Appendix for explicit solutions of problems \eqref{e2.1} and \eqref{e2.2}.
\end{remark}

\section{Trajectory of free piston and delta shock}\label{sec3}
This section is devoted to proving Theorems \ref{t2.1}-\ref{thm2.3}, by deriving the ODE satisfied by the trajectory of the piston and delta shock. For the piston problem, we will use two methods. One is a direct application of conservation of mass and momentum to a control volume of sticking particles and piston, which is essentially from the microscopic point of view and physically favorable; the other is to utilize the concept of Radon measure solutions of initial-boundary value problems of pressureless Euler equations, which is a macroscopic approach. For the singular Riemann problem, the delta shock is derived straightforwardly from Definition \ref{def1}.

\subsection{Motion of piston driven by sticky particles}\label{sec31}

Set $(\tilde{\rho}_1, \tilde{u}_1)^{\top}$ {and} $(\tilde{\rho}_2,\tilde{u}_2)^\top$ be the constant state {on the} left {and} right side of the piston respectively.
For any given $t>0$,
assume that during the period $[0,t]$, the gas in the intervals $(a(t),0)$~and~$(l,b(t))$~have sticken respectively to the left and right surface of the piston.
Then
\begin{equation}\label{2.3}
|a(t)|  = \tilde{u}_1t-x_{1}(t),\quad
b(t)=x_{2}(t)-\tilde{u}_2 t.	
\end{equation}
Let $m_{1}(t)$ and $m_{2}(t)$ be respectively the mass concentrated on the left
and right surface of the piston from  time $0$ to $t$.  Then by \eqref{2.3}, and {recalling}   $x_{2}(t)=x_{1}(t)+l$, we have
\begin{equation}\label{2.5}
\left \{
\begin{split}
&m_{1}(t)=\tilde{\rho}_1|a(t)|=\tilde{\rho}_1\tilde{u}_1t-\tilde{\rho}_1x_{1}(t),\\
&m_{2}(t)=\tilde{\rho}_2(b(t)-l)=-\tilde{\rho}_2 \tilde{u}_2t+\tilde{\rho}_2x_{1}(t).
\end{split} \right.
\end{equation}
By conservation of momentum,
\begin{equation}\label{2.6}
\begin{split}
(m_{1}(t)+m_{0}+m_{2}(t))x_{1}'(t)=m_{1}(t)\tilde{u}_1+m_{0}\cdot u_0+m_{2}(t)\tilde{u}_2,
\end{split}
\end{equation}
which, together with \eqref{2.5}, implies that
\begin{equation}\label{2.7}
\begin{split}
&x_{1}'(t)\big(\tilde{\rho}_1\tilde{u}_1t-\tilde{\rho}_1x_{1}(t)+m_{0} -\tilde{\rho}_2\tilde{u}_2t+\tilde{\rho}_2x_{1}(t)\big)\\
&\qquad =\tilde{\rho}_1\tilde{u}_1^{2}t-\tilde{\rho}_1\tilde{u}_1x_{1}(t) -\tilde{\rho}_2\tilde{u}_2^{2}t+\tilde{\rho}_2\tilde{u}_2x_{1}(t)+m_0u_0.
\end{split}
\end{equation}
Taking a derivative with respect to $t$ on both sides, we have
\begin{equation}\label{2.8}
\begin{split}
x_{1}''(t)&+\dfrac{(\tilde{\rho}_1-\tilde{\rho}_2)x_{1}'(t)-2(\tilde{\rho}_1\tilde{u}_1 -\tilde{\rho}_2\tilde{u}_2)}{(\tilde{\rho}_1-\tilde{\rho}_2)x_{1}(t)-(\tilde{\rho}_1\tilde{u}_1 -\tilde{\rho}_2\tilde{u}_2)t-m_{0}}x_{1}'(t)\\
&\qquad +\dfrac{\tilde{\rho}_1\tilde{u}_1^{2}-\tilde{\rho}_2\tilde{u}_2^{2}}{(\tilde{\rho}_1 -\tilde{\rho}_2)x_{1}(t)-(\tilde{\rho}_1\tilde{u}_1-\tilde{\rho}_2\tilde{u}_2)t-m_{0}}=0.	\end{split}
\end{equation}	
For $m_0>0$, \eqref{2.7} and $x_1(0)=0$ show that $x_1'(0)=u_0$. We thus derived problem \eqref{e2.1}. For \eqref{e2.2}, just taking $\tilde{\rho}_2=0$ in the above calculations.

\subsection{Trajectory of piston by Radon measure solutions}\label{s2.2}
To solve the free piston as a problem of partial differential equations, we need a proper definition of Radon measure solutions, similar to that of Definition \ref{def1}, with the boundary condition \eqref{2.2} taken into account.

\begin{definition}
Let $x=x_1(t)$ be a Lipschitz function, and $\Omega_1$ defined by \eqref{eq6add}.  For the piston problem \eqref{1.1}\eqref{1.2}\eqref{2.2} in $\Omega_1$, assume that $\varrho$ 
is a Radon measure on ${\bar\Omega_1}$, $u$ is a $\varrho$-measurable function and $w_p(t)\in L^{1}_{\mathrm{loc}}([0,+\infty))$ {is nonnegative}.
We call ${(\varrho,u)^\top}$  a (Radon) measure solution to this problem, provided  that there exist Radon measures $m$ and $n$ satisfying the following: 
\par i) for any $\phi(x,t)\in C_{c}^{1}(\Bbb R^2),$  there hold
\begin{align}\label{2.19}
&\int_{\Omega_1}\phi_t(x,t)\,\dd \varrho(x,t)+\int_{\Omega_1} \phi_x(x,t)\,\dd m(x,t)+ \int_{-\infty}^0\phi(x,0)\,\dd\varrho(x,0) =0,\\
&\int_{\Omega_1}\phi_t(x,t)\,\dd m(x,t)+\int_{\Omega_1}\phi_x(x,t)\,\dd n(x,t)-\int_{0}^{+\infty} w_p^1(t)\phi(x_1(t),t)\,\dd t\nonumber
\\&\qquad\qquad\qquad\qquad\qquad\qquad\qquad\qquad\qquad\quad+\int_{-\infty}^{0} \phi(x,0)\,\dd m(x,0)=0;\label{220}
\end{align}
\par ii) $\varrho$ is nonnegative, while $m$ and  $n$ are absolutltly continuous with respect to $\varrho$. The Radon-Nikodym derivatives satisfy
\begin{equation}\label{2.22}
\begin{split}
&u(x,t)\doteq\dfrac{\dd m(x,t)}{\dd \varrho(x,t)}, \qquad u(x,t)^2=\dfrac{\dd n(x,t)}{\dd\varrho(x,t)}.
\end{split}
\end{equation}
\end{definition}
\qed

We remark that the function $w^1_p(t)$ represents the pressure (hence force) on the left side of the piston at time $t$, exerted due to the inertia of pressureless particles. This is the key point that differs essentially from the Cauchy problem of the pressureless gas dynamics. Similarly we could formulate the definition of piston problem in $\Omega_2$ defined by \eqref{eq7add}, just replacing $x_1(t)$ by $x_2(t)$, $w_p^1$ by $w_p^2$, and the minus sign in front of $w_p^1$ by plus sign, as well as integrating the initial data on  $(l,+\infty)$ instead of $(-\infty,0)$. The piston moves according to Newton's second law, given by \eqref{eq5}.

Let $\mathcal{L}^{2}$ be the Lebesgue measure on $\Bbb R^{2}$. Suppose that
\begin{equation}\label{2.24}
\begin{split}
&\varrho=\tilde{\rho}_1\mathcal{L}^{2}\lfloor\Omega_{1} +\alpha_{1}(t)\delta_{\Gamma_{1}},\\
&m=\tilde{\rho}_1\tilde{u}_1\mathcal{L}^{2}\lfloor\Omega_{1} +w_{m}^{1}(t)\delta_{\Gamma_{1}},\\
&n=\tilde{\rho}_1\tilde{u}_1^{2}\mathcal{L}^{2}\lfloor\Omega_{1}  +w_{n}^{1}(t)\delta_{\Gamma_{1}},
\end{split}
\end{equation}
where $\alpha_{1}(t),~w_{m}^{1}(t),~w_{n}^{1}(t)$ are functions to be determined. Notice that the weighted Dirac measure $\eta\doteq\alpha_{1}(t)\delta_{\Gamma_{1}}$ supported on the curve $\Gamma_1=\{(x_1(t), t): t\ge0\}$ is defined by
$$\int_{\Omega_1}\phi(x,t)\,\dd\eta(x,t)=\langle\alpha_{1}(t)\delta_{\Gamma_{1}}, \phi\rangle=\int_{0}^{\infty}\alpha_1(t)\phi(x_1(t),t)\,\dd t.$$
Now substituting \eqref{2.24} into 
\eqref{2.19}, it follows
\begin{equation}\label{2.25}
\begin{split}
\int_{\Omega_{1}}{\tilde{\rho}_1\partial_{t}\phi}\,\dd x\dd t& +\int_{0}^{+\infty}\alpha_{1}(t)\partial_{t}\phi(x_1(t),t)\, \dd t +\int_{\Omega_{1}}{\tilde{\rho}_1\tilde{u}_1\partial_{x}\phi}\,\dd x\dd t \\
& +\int_{0}^{+\infty}w_{m}^{1}(t)\partial_{x_{1}}\phi(x_1(t),t)\, \dd t+\int_{-\infty}^{0}{\rho(0,x)\phi(0,x)}\,\dd x=0.
\end{split}
\end{equation}
Observing that
\begin{equation}\label{2.26}
\begin{split}
&\int_{0}^{+\infty}\alpha_{1}(t)\partial_{t}\phi (x_1(t),t)\, \dd t =-\alpha_{1}(0)\phi(0,0)\\&\quad-\int_{0}^{+\infty}\alpha_{1}(t)\partial_{x_{1}}\phi(x_1(t),t) x_{1}'(t)\,\dd t-\int_{0}^{+\infty}\phi(x_1(t),t) \dfrac{\dd\alpha_{1}(t)}{\dd t}\,\dd t,
\end{split}
\end{equation}
while integration-by-parts yields
\begin{equation}\label{2.27}
\begin{split}
&\int_{\Omega_{1}}{\tilde{\rho}_1\partial_{t}\phi}\,\dd x\dd t +\int_{\Omega_{1}}{\tilde{\rho}_1\tilde{u}_1\partial_{x}\phi}\,\dd x\dd t +\int_{-\infty}^{0}{\rho(0,x)\phi(0,x)}\,\dd x\\
&\qquad +\int_{0}^{+\infty}w_{m}^{1}(t) \partial_{x_{1}}\phi(x_1(t),t) \,\dd t
=-\int_{0}^{+\infty}\tilde{\rho}_1x_{1}'(t)\phi(x_1(t),t)\, \dd t\\&\qquad\quad+\int_{0}^{+\infty}\tilde{\rho}_1\tilde{u}_1\phi(x_1(t),t)\, \dd t+\int_{0}^{+\infty}w_{m}^{1}(t)\partial_{x_{1}}\phi(x_1(t),t)\, \dd t,
\end{split}
\end{equation}
thus \eqref{2.25} implies that
 \begin{align}\label{2.28}
&-\alpha_{1}(0)\phi(0,0)-\int_{0}^{+\infty}\alpha_{1}(t)\partial_{x_{1}}\phi(x_1(t),t) x_{1}'(t)\,\dd t-\int_{0}^{+\infty}\phi(x_1(t),t) \dfrac{\dd\alpha_{1}(t)}{\dd t}\,\dd t\nonumber\\
=&\int_{0}^{+\infty}\tilde{\rho}_1x_{1}'(t)\phi(x_1(t),t) \,\dd t-\int_{0}^{+\infty}\tilde{\rho}_1\tilde{u}_1\phi (x_1(t),t) \,\dd t\\
&\qquad\qquad-\int_{0}^{+\infty}w_{m}^{1}(t)\partial_{x_{1}}\phi(x_1(t),t) \,\dd t.\nonumber
 \end{align}
{By the} arbitrariness of $\phi$, it follows
 \begin{equation}\label{2.29}
 \left \{
 \begin{split}
 &\dfrac{\dd\alpha_{1}(t)}{\dd t}=-\tilde{\rho}_1x_{1}'(t)+\tilde{\rho}_1\tilde{u}_1,\qquad \alpha_{1}(0)=0,\\
 &w_{m}^{1}(t)=\alpha_{1}(t)\cdot x_{1}'(t).
 \end{split}\right.
 \end{equation}
Similarly, by substituting \eqref{2.24} into \eqref{220}, we  get 
\begin{equation}\label{2.24new}
 \left \{
 \begin{split}
 &w_{m}^{1}(t)=0,\\
 &\dfrac{\dd w_{m}^{1}(t)}{\dd t}=-w_{p}^{1}(t)-\tilde{\rho}_1\tilde{u}_1x_{1}'(t) +\tilde{\rho}_1\tilde{u}_1^{2},\\
 &w_{n}^{1}(t)=w_{m}^{1}(t)\cdot x_{1}'(t).
 \end{split}\right.
 \end{equation}

For the piston problem in $\Omega_2$, supposing that
\begin{equation*}
\begin{split}
&\varrho=\tilde{\rho}_2\mathcal{L}^{2}\lfloor\Omega_{2} +\alpha_{2}(t)\delta_{\Gamma_{2}},\\
&m=\tilde{\rho}_2\tilde{u}_2\mathcal{L}^{2}\lfloor{\Omega_{2}}+w_{m}^{2}(t)\delta_{\Gamma_{2}},\\
&n=\tilde{\rho}_2\tilde{u}_2^{2}\mathcal{L}^{2}\lfloor{\Omega_{2}}+ w_{n}^{2}(t)\delta_{\Gamma_{2}},
\end{split}
\end{equation*}
where $\Gamma_2=\{(x_2(t), t): t\ge0\}$, similar calculations show that
 \begin{equation}\label{2.30}
 \left \{
 \begin{split}
 &\alpha_{2}(0)=0,\\
 &{\frac{\dd\alpha_2(t)}{\dd t} =\tilde{\rho}_2x_{2}'(t)-\tilde{\rho}_2\tilde{u}_2,}\\
 &w_{m}^{2}(t)=\alpha_{2}(t)\cdot x_{2}'(t),\\
 &\dfrac{\dd w_{m}^{2}(t)}{\dd t}=w_{p}^{2}(t)+\tilde{\rho}_2\tilde{u}_2x_{2}'(t)-\tilde{\rho}_2\tilde{u}_2^{2},\\
 &w_{n}^{2}(t)=w_{m}^{2}(t)\cdot x_{2}'(t).
  \end{split}\right.
 \end{equation}
{Combining \eqref{2.29}-\eqref{2.30} gives}
\begin{equation}\label{2.31}
\left \{
\begin{split}
&\alpha_{1}(t)=-\tilde{\rho}_1x_{1}(t)+\tilde{\rho}_1\tilde{u}_1t,\\
&w_{m}^{1}(t)=[-\tilde{\rho}_1x_{1}(t)+\tilde{\rho}_1\tilde{u}_1t]\cdot x_{1}'(t),\\
&w_{p}^{1}(t)=x_{1}''(t)[\tilde{\rho}_1x_{1}(t)-\tilde{\rho}_1\tilde{u}_1t]
+x_{1}'(t)[-2\tilde{\rho}_1\tilde{u}_1+\tilde{\rho}_1x_{1}'(t)]+\tilde{\rho}_1\tilde{u}_1^{2},\\
&\alpha_{2}(t)=\tilde{\rho}_2x_{2}(t)-\tilde{\rho}_2\tilde{u}_2t-\tilde{\rho}_2l,\\
&w_{m}^{2}(t)=[\tilde{\rho}_2x_{2}(t)-\tilde{\rho}_2\tilde{u}_2t]\cdot x_{2}'(t),\\
&w_{p}^{2}(t)=x_{2}''(t)[\tilde{\rho}_2x_{2}(t)-\tilde{\rho}_2\tilde{u}_2t-\tilde{\rho}_2l]+x_{2}'(t)[-2\tilde{\rho}_2\tilde{u}_2
+\tilde{\rho}_2x_{2}'(t)]+\tilde{\rho}_2\tilde{u}_2^{2}.
\end{split}\right.
\end{equation}

Recalling that $x_{2}(t)=x_{1}(t)+l, \forall \,t\ge0$, then thanks to \eqref{2.31}, the Newton's second law \eqref{eq5}, i.e.,
 \begin{equation*}
  x_{1}''(t)=x_{2}''(t)=\dfrac{w_{p}^{1}(t)-w_{p}^{2}(t)}{m_{0}},
 \end{equation*}
becomes
\begin{equation}\label{2.34}
 \begin{split}
  x_{1}''(t)=&\dfrac{x_{1}''(t)[(\tilde{\rho}_1-\tilde{\rho}_2)x_{1}(t)-(\tilde{\rho}_1\tilde{u}_1-\tilde{\rho}_2\tilde{u}_2)t]}{m_{0}}\\
  &+\dfrac{x_{1}'(t)[-2(\tilde{\rho}_1\tilde{u}_1\tilde{\rho}_2\tilde{u}_2)+x_{1}'(t)(\tilde{\rho}_1\tilde{\rho}_2)]}{m_{0}}+\dfrac{\tilde{\rho}_1\tilde{u}_1^{2}\tilde{\rho}_2\tilde{u}_2^{2}}{m_{0}},
 \end{split}
 \end{equation}
which in turn could be simplified to the ODE in \eqref{e2.1}. 	
This completes the proof of Theorem \ref{t2.1}.

The proof of Theorem \ref{t2.2} follows the same procedure shown above, just taking
$$\tilde{\rho}_2=\begin{cases} 0, &\text{if}~x_2(t)<x<u_2t+l,\\
\rho_2, &\text{if}~ x\ge u_2t+l,
\end{cases}$$
and $\alpha_2=0$, $w_p^2=0$.

\begin{remark}
Notice that by some computations (cf. the expressions of $x_1(t)$ in the Appendix), one can show that the over-compressibility entropy condition guarantees that $w_p^1$ and $w_p^2$ are nonnegative for the pressureless Euler flow. 
\end{remark}

\subsection{Delta shock of singular Riemann problem}\label{s2.3}
The proof of Theorem \ref{thm2.3} is still based upon direct calculations according to Definition \ref{def1}, by supposing that
\begin{equation}\label{2.45}
\begin{split}
&\varrho(t)=\tilde{\rho}_1\mathcal{L}^{1}\lfloor\{x: x<x_1(t)\}+\tilde{\rho}_2\mathcal{L}^{1}\lfloor\{x: x>x_1(t)\} +\alpha(t)\delta_{x_1(t)},\\
&m(t)=\tilde{\rho}_1\tilde{u}_1\mathcal{L}^{1}\lfloor\{x: x<x_1(t)\} +\tilde{\rho}_2\tilde{u}_2\mathcal{L}^{1}\lfloor\{x: x>x_1(t)\}+w_{m}(t)\delta_{x_1(t)},\\
&n(t)=\tilde{\rho}_1\tilde{u}_1^{2}\mathcal{L}^{1}\lfloor\{x: x<x_1(t)\} +\tilde{\rho}_2\tilde{u}_2^{2}\mathcal{L}^{1}\lfloor\{x: x>x_1(t)\}+w_{n}(t)\delta_{x_1(t)},
\end{split}
\end{equation}
where $\alpha(t),\, w_{m}(t)$ and $w_{n}(t)$ are functions to be determined, with initial data $\alpha(0)=m_0$, $w_m(0)=m_0u_0$ provided by the singular Riemann initial data \eqref{eq6}. 
Substituting \eqref{2.45} into \eqref{2.43}, {by the} arbitrariness of $\phi$,~we have the following generalized Rankine-Hugoniot conditions \cite[(3.3.10) in p. 60]{lizhangyang}
 \begin{equation}\label{2.46}
 \left \{
 \begin{split}
 &\dfrac{\dd\alpha(t)}{\dd t}
 =-\tilde{\rho}_1x_{1}'(t)+\tilde{\rho}_1\tilde{u}_1+\tilde{\rho}_2x_{1}'(t) -\tilde{\rho}_2\tilde{u}_2,\\
 &\dfrac{\dd w_{m}(t)
 }{\dd t}=-\tilde{\rho}_1\tilde{u}_1x_{1}'(t) +\tilde{\rho}_1\tilde{u}_1^{2}+\tilde{\rho}_2\tilde{u}_2x_{1}'(t) -\tilde{\rho}_2\tilde{u}_2^{2},\\
 &w_{m}(t)=\alpha(t)x_{1}'(t).
 \end{split}\right.
 \end{equation}
Integration of the first two equations {yields}
\begin{equation}\label{2.47}
\left \{
\begin{split}
&\alpha(t)=-\tilde{\rho}_1x_{1}(t)+\tilde{\rho}_1\tilde{u}_1t +\tilde{\rho}_2x_{1}(t)-\tilde{\rho}_2\tilde{u}_2t+m_{0},\\
&w_{m}(t)
=-\tilde{\rho}_1\tilde{u}_1x_{1}(t)+\tilde{\rho}_1\tilde{u}_1^{2}t +\tilde{\rho}_2\tilde{u}_2x_{1}(t)-\rho_{2}\tilde{u}_2^{2}t+m_0u_0.
\end{split}\right.
\end{equation}
Then  by the last equation in \eqref{2.46}, we have
\begin{equation}\label{2.48}
\begin{split}
[-\tilde{\rho}_1x_{1}(t)+\tilde{\rho}_1\tilde{u}_1t&+\tilde{\rho}_2x_{1}(t) -\tilde{\rho}_2\tilde{u}_2t+m_{0}]x_{1}'(t)=\\
&-\tilde{\rho}_1\tilde{u}_1x_{1}(t)+\tilde{\rho}_1\tilde{u}_1^{2}t +\tilde{\rho}_2\tilde{u}_2x_{1}(t)-\tilde{\rho}_2\tilde{u}_2^{2}t+m_0u_0,
\end{split}
\end{equation}	
which is exactly \eqref{2.7}.  Then the claims of Theorem \ref{thm2.3} follow easily.

\begin{remark}\label{rem3}
The limiting case $m_0=0$ is interesting, for which $u_0=x_1'(0)$ could not be prescribed arbitrarily, and \eqref{eq5} does not make sense. However, \eqref{2.7} derived from the approach of delta shock still works. By L'Hospital rule, from \eqref{2.6} we may solve that
$u_0=\frac{\tilde{u}_1+\tilde{u}_2}{2}$ for $\tilde{\rho}_1=\tilde{\rho}_2$, and for $\tilde{\rho}_1\ne\tilde{\rho}_2$, there are two solutions:
$$u_0=\frac{\big(\tilde{\rho}_2\tilde{u}_2-\tilde{\rho}_1\tilde{u}_1\big)\pm\sqrt{\tilde{\rho}_1\tilde{\rho}_2}|\tilde{u}_1-\tilde{u}_2|}{\tilde{\rho}_2-\tilde{\rho}_1}
=\frac{\sqrt{\tilde{\rho}_2}\tilde{u}_2\pm \sqrt{\tilde{\rho}_1}\tilde{u}_1}{\sqrt{\tilde{\rho}_2}\pm\sqrt{\tilde{\rho}_1}}.$$
To be consistent with the limiting case $\tilde{\rho}_1=\tilde{\rho}_2$, in the above we shall take the `$+$' sign. It is seen that the additional requirement of prescribing  $x'(0)$ is due to the derivative we took on \eqref{2.7}. We also note that Theorems \ref{t3.1} and \ref{t3.2} below on the solutions of free piston problem are also valid for the case $m_0=0$. This is a great benefit of considering the free piston as a delta shock.
\end{remark}

\begin{remark}
The equivalence of free piston  and  delta shock also holds for compressible Euler equations with general pressure law, such as polytropic gas, which is discussed in \cite{jinquyuan2021riemann}.
\end{remark}

\appendix
\section{Solvability of piston's trajectories}\label{s3}

For convenience of readers, in this Appendix we list the results on the six cases of the free piston problem. We omit the proof, since the details are the same as those in \cite[pp. 67-72]{lizhangyang}.

\subsection{Case 1: $u_2<u_0<u_{1}$}
\begin{theorem}\label{t3.1}
Suppose that $u_2<u_0<u_{1}$, then the free piston problem \eqref{1.1}-\eqref{eq5} admits a global solution, in which the trajectory of the piston satisfies \eqref{e2.1}. The velocity of the piston $x'_ {1}(t)\in(u_2,u_1)$ for any $t>0$, and it converges to a constant as $t\to+\infty$.  There are concentration on both sides of the piston for all the time. Particularly, we have \\
1) if ~$\rho_{1}\neq\rho_{2}$,
\begin{align}\label{3.1}
 x_{1}(t)=&\dfrac{-\sqrt{\rho_{1}\rho_{2}(u_{1}-u_{2})^{2}t^{2}+2m_{0}[\rho_{1}(u_{1}-u_{0})-\rho_{2}(u_{2}-u_{0})]t+m_{0}^{2}}}{\rho_{1}-\rho_{2}}\\
 &+\dfrac{(\rho_{1}u_{1}-\rho_{2}u_{2})t+m_{0}}{\rho_{1}-\rho_{2}},\\
\label{3.18a}
\lim_{t\rightarrow+\infty} x_{1}'(t)=&\dfrac{-\sqrt{\rho_{1}\rho_{2}}(u_{1}-u_{2})+\rho_{1}u_{1} -\rho_{2}u_{2}}{\rho_{1}-\rho_{2}}=\frac{\sqrt{\rho}_1u_1+\sqrt{\rho_2}u_2}{\sqrt{\rho}_1+\sqrt{\rho_2}};
 \end{align}
2) if $\rho_{1}=\rho_{2}$,
\begin{align}\label{3.2}
x_{1}(t)=&\frac{u_1+u_2}{2}t+\frac{m_0^2(u_1+u_2 -2u_0)}{2\rho_1(u_1-u_2)[\rho_{1}(u_{1}-u_{2})t+m_{0}]} -\frac{m_0(u_1+u_2-2u_0)}{2\rho_1(u_1-u_2)},
\\
\label{3.33}
\lim_{t\rightarrow+\infty} x_{1}'(t)=&\dfrac{u_{1}+u_{2}}{2}.
 \end{align}
\end{theorem}

\subsection{Cases 2 and 3: $u_{1}>u_0$, $u_{2}>u_0$}
For Case 3, we know the physical picture is as follows:  there is a moment $t_1$, such that if $0< t<t_1$, there is a vacuum near the right side of the piston; as the flow on the left accelerates the piston,  it chases the constant moving flow $U_2=(\rho_2,u_2)^\top$ on the right, and at $t_1$ the piston meets $U_2$ and then concentration on the right side of the piston occurs; what happens next is the same as in Case 1. For Case 2, the story is a little bit different: the piston could never catch up with the right state $U_2$ and there is always vacuum near its right side. The following theorem verifies this.
\begin{theorem}\label{t3.2}
For the free piston problem \eqref{1.1}-\eqref{eq5}, suppose that $u_{1}>u_0$, $u_{2}>u_0$, then it admits a global solution.

More specifically, for Case 2: $u_0<u_{1}<u_{2}$, there is always vacuum near the right surface of the piston, namely the piston will never catch up with the right-side moving  flows with state $(\rho_2, u_2)^{\top}$, and the trajectory  of the left surface of the piston  is
\begin{equation}\label{3.36}
 x_{1}(t)=\dfrac{-\sqrt{2m_{0}\rho_{1}(u_{1}-u_{0})t+m_{0}^{2}}}{\rho_{1}}+u_{1}t +\dfrac{m_{0}}{\rho_{1}}.
 \end{equation}
The limiting velocity of the piston is
\begin{equation}\label{3.36a}
 \lim_{t\rightarrow+\infty}x_{1}'(t)=u_{1}.
 \end{equation}

For Case 3: $u_0<u_{2}<u_{1}$, $x_1(t)$ is given by
\begin{equation}\label{3.36a}
 x_{1}(t)=\begin{cases}
           \dfrac{-\sqrt{2m_{0}\rho_{1}(u_{1}-u_{0})t+m_{0}^{2}}}{\rho_{1}}+u_{1}t+\dfrac{m_{0}}{\rho_{1}}, &\mbox{if~}0\leq t\leq t_{1},\\
           \tilde{x}_1(t-t_1)+u_2t_1,&\mbox{if~}t_1<t,
          \end{cases}
 \end{equation}
where $\displaystyle t_1\doteq\frac{2m_0(u_2-u_0)}{\rho_1(u_1-u_2)^2}$, and $\tilde{x}_1(\cdot)$ is the function given by \eqref{3.1} if $\rho_1\neq \rho_2$, and by \eqref{3.2} if $\rho_1= \rho_2$, while in both of which,  $m_0$ is replaced by $\alpha (t_1)$ defined in \eqref{2.47} and $u_0$ is replaced by $\displaystyle \frac{2u_1u_2-u_0u_2-u_0u_1}{u_1+u_2-2u_0}$. The limiting velocity of the piston is given by \eqref{3.18a} and \eqref{3.33} respectively.
\end{theorem}

\subsection{The other cases}
As discussed in Section \ref{sec23}, Case 4  and Case 5 ($u_{1}<u_0$, $u_{2}<u_0$) are in essence the same as Cases 2 and 3 studied above. So we will not repeat the results any more.
For the last Case 6, namely $u_{1}\leq u_0\leq u_{2}$,
the pressureless gas moves uniformly away from the piston, and there are vacuum near the piston and the piston's trajectory is given by $x_1(t)=u_0t,\ \forall \, t\ge0$.
The solution for this case is trivial.

\begin{acknowledgements}

The authors are grateful to Professor Jiequan Li, for pointing out the connection between delta shock and free piston problem in a private communication, which consists the main contribution of this manuscript. He also gave us valuable and detailed suggestions and corrections on an earlier draft. We thank him sincerely for his generosity and encouragements!
\end{acknowledgements}

%
%



\end{document}